\documentclass[12pt,twoside,openany]{paper}
\usepackage{enumerate,amsmath,graphics,amssymb,graphicx,amscd,xypic,amscd,amsbsy,multirow,float,booktabs,verbatim}
\newcommand{\QED}{\hspace*{\fill}\rule{2.5mm}{2.5mm}}
\newtheorem{theorem}{Theorem}[section]
\newtheorem{definition}{Definition}[section]

\newenvironment{proof}{\noindent{\bf Proof\ }}{\QED\\}
\newcommand{\R}{\mathbb{R}}

\newcommand{\Q}{\mathbb{Q}}

\newtheorem{lemma}{Lemma}[section]

\begin{document}
\begin{center}
\vspace{0.5cm} {\large \bf ``Closure of random samples"}\\
\vspace{1cm} Reza Hosseini, Simon Fraser University\\
Statistics and actuarial sciences, 8888 University Road,\\
Burnaby, BC, Canada, V65 1S6\\
reza1317@gmail.com
\end{center}

\begin{abstract}
In this paper we show that the closure of a random sample for a
$k$-dimensional random vector is almost surely a deterministic set
of all heavy points of the distribution. A heavy point is defined to
be a point for which all its neighborhoods have positive
probability.
\end{abstract}

\noindent Keywords: Distribution function; random sample; closure;
almost surely

\section{Introduction}

Although a random variable is deterministic and random by
definition, one can find deterministic features specially when an
infinite random sample is available. A deterministic feature can be
an event with probability 1. Kolmogorov 0-1 Lemma (See
\cite{breiman} for example) provides a general result for finding
events with probability 1. One can also look for random sets that
are almost surely deterministic. This paper finds one of the
simplest of such almost sure deterministic sets. We show that the
closure of a random sample for a $k$-dimensional random vector is
almost surely the deterministic set of all heavy points of the
distribution. A heavy point is defined to be a point for which all
its neighborhoods have positive probability.

\section{Closure of random samples}

Suppose $(\Omega, \Sigma, P)$ a probability space and  $X: \Omega
\rightarrow \R^k$ a random vector with respect to the Borel
sigma-algebra on $\R^k$ with distribution function $F_X$. We begin
by a definition.

\begin{definition}
$x \in \R^k$ is called a heavy point of a distribution function
$F_X$ if
\[P(X \in U)>0,\]
if $U$ is an open set and includes $x$. We call the set of all such
points the heavy set of $X$ and denote it by $H(X)$. Also we call
the $N(X)=\R^k-H(X)$ the weightless set or null set of $X$. It is
clear that these sets only depend on the distribution function.
\end{definition}
In the following, we prove a lemma about the properties of such
sets.

\begin{lemma}
Suppose $X$ and $F_X$ are defined as above. Then\\
a) $N(X)$ is open.\\
b) $H(X)$ is closed.\\
c) $P(X \in N(X))=0.$
\end{lemma}

\begin{proof}
We will denote an open ball of radius $\epsilon>0$ around $x$ by
$Ba(x,\epsilon)$.\\
a) For every $x \in N(X)$ by definition of $H(X)$ there exist and
$\epsilon>0$ such that
\[P(X \in Ba(x,\epsilon))=0.\]
But then it is obvious that any $y \in Ba(x,\epsilon)$ also belongs
to $N(X)$ (simply find $\epsilon'$ such that $Ba(y,\epsilon') \subset Ba(x,\epsilon)$) and the proof is complete.\\
b) Straight forward corollary of a)\\
c) Let $Q'=\Q \cap N(X)=\{q_1,q_2,\cdots\}$ and
\[r_i'=\sup_{r \geq 0} P(X \in Ba(q_i,r))=0.\]
Such $r_i'$ is positive since $q_i \in N(X)$. Then let
$r_i=\frac{r_i'}{2}<r_i'$. We have $P(X \in Ba(q_i,r_i))=0.$ Let
\[N=\cup_{i=1}^{\infty} Ba(q_i,r_i) \subset N(X).\]
Since $N$ is a countable union of sets with zero probability, it's
probability is also zero. It only remains to show that $N(X) \subset
N$. Take $x \in N(X)$ then there exists $r>0$ such that
\[P(X \in Ba(x,r))=0.\] Take a rational number in $Ba(x,r)$ such
that $|x-q|<\frac{r}{4}$ then $q=q_i \in Q'$ and
$Ba(q_i,\frac{3r}{4}) \subset Ba(x,r)$ and hence \[r_i' \geq
\frac{3r}{4} \Rightarrow r_i \geq \frac{3r}{8}> \frac{r}{2}.\] Hence
$x \in Ba(q_i,r_i) \subset N.$
\end{proof}

Before to prove the main theorem we point out an  important property
of closed subsets of $\R^k$. A topological space is called separable
if it posses a countable dense subset. It is well-known that any
topological space with a countable basis is separable. See
\cite{kelly} page 49 for example.

\begin{lemma}
Every closed subset of $\R^k$ is separable.
\end{lemma}

\begin{proof}
We offer two proofs one based on the properties of separable spaces
mentioned above and an elementary proof. Suppose $E$ is closed in $\R^k$.\\
Proof 1. $E$ admits a countable basis consisting of all intervals with rational endpoints in $E$. Then the proof is complete since every topological
space with a countable basis is separable.\\
Proof 2. Consider the $k$-dimensional rational space $\Q^k$ and let
\[Q_n=\{Ba(q,1/n),\;q \in \Q^k\},\;n=1,2,\cdots\]
and for every $n$ we define $D_n$ as follows. If $E \cap Ba(q,1/n)$
is non-empty choose one element of this set $x_{q,1/n}$ and put in
$D_n$. Hence $D_n$ is countable and so is
$D=\cup_{n=1}^{\infty}D_n$. We also claim $D$ is dense in $E$.
Suppose $e \in E$ then it is obvious that for every $\epsilon>0$
there exist $q,n$ such that
\[Ba(q,1/n) \subset Ba(x,\epsilon)\]
but by construction we picked an element of $E$ from $Ba(q,1/n)$.
\end{proof}

\begin{lemma}
 Suppose $X$ a $k$-dim random vector, $x \in H(X)$ and $X_1,X_2,\cdots$ a random sample from $X$. Let $E=\{X_1,X_2,\cdots\}$
 then $x \in \overline{E}, \;\;a.s..$
\end{lemma}

\begin{proof}
 Let $P(X \in Ba(x,1/n))=p_n>0$ and \[A_n=\{\omega \in \Omega| \exists i,\;\;X_i(\omega) \in Ba(x,1/n)\}.\]
 Then \[A=\cap_{i=1}^{\infty}A_i=\{\omega \in \Omega|\exists i,\;\; X_i(\omega) \in
Ba(x,1/n),\;\;n=1,2,\cdots\}.\] We need to show $P(A)=1.$\\
\[P(A_n)=P(\exists i,\;X_i \in Ba(x,1/n))=1-P(\forall i,\;X_i \notin Ba(x,1/n))=\]
\[1-\prod_{i=1}^{\infty}P(X_i \notin B(x,1/n))=1-(1-p_n)^{\infty}=1-0.\]
Now $P(A)=\lim_{n \rightarrow \infty} P(A_n)=1.$

\end{proof}

\begin{theorem}
Suppose $X_1,X_2,\cdots$ be a random sample from $X$ a
$k$-dimensional random vector. Let $E=\{X_1,X_2,\cdots\}$ and
suppose $\overline{E}$ is the closure of $E$ with respect to the
Euclidean distance topology of $\R^k$. Then $\overline{E}=H(X),\;\;$
almost surely.
\end{theorem}

\begin{proof}
(i) $H(X) \subset \overline{E},\;a.s.$: Since $H(X)$ is closed, it
is separable and take a dense countable subset of $H(X)$,
\[H'=\{h_1,h_2,\cdots\}.\]
Let $C_i=\{\omega \in \Omega|h_i \in \overline{E}\}$ and
$C=\cap_{i=1}^{\infty} C_i$. Then $P(C_i)=1$ by the previous lemma
for each $i=1,2,\cdots.$ But then $P(C)=1$ which means $H' \subset
\overline{E},\;a.s.$ and hence
$H=\overline{H'} \subset \overline{E}$.\\
(ii) $\overline{E} \subset H(X),\;$ almost surely. We only need to
show $E \subset H(X),\;$ almost surely. But we showed $P(X_i \in
H(X))=1-P(X \in N(X))=1-0,\;i=1,2,\cdots$. Let $A_i=\{\omega \in
\Omega | X_i \in H(X)\}$ and $A=\cap_{i=1}^{\infty} A_i$. Then $A=E
\subset H(X)$ and
\[P(A)=\lim_{n \rightarrow \infty} P(A_n)=1.\]
\end{proof}

\bibliographystyle{plain}
\bibliography{mybibreza}
\end{document}